\documentclass{amsart}

\usepackage{amssymb}

\usepackage{epsfig}

\newtheorem{thm}{Theorem}

\newtheorem{lem}[thm]{Lemma}
\newtheorem{cor}[thm]{Corollary}

\newtheorem{prop}[thm]{Proposition}

\theoremstyle{definition}
\newtheorem{defn}[thm]{Definition}

\newtheorem{say}[thm]{}
\newtheorem{exmp}[thm]{Example}

\newtheorem{ques}[thm]{Question}    

\newtheorem{rem}[thm]{Remark}

\newtheorem*{ack}{Acknowledgments}       
\newtheorem{notation}[thm]{Notation}

\newtheorem{defn-thm}[thm]{Definition--Theorem}  
\newtheorem{defn-lem}[thm]{Definition--Lemma}  

\theoremstyle{remark}

\let \cedilla =\c
\renewcommand{\c}[0]{{\mathbb C}}

\renewcommand{\o}[0]{{\mathcal O}}
\newcommand{\z}[0]{{\mathbb Z}}

\renewcommand{\r}[0]{{\mathbb R}}

\newcommand{\p}[0]{{\mathbb P}}

\newcommand{\q}[0]{{\mathbb Q}}
\newcommand{\map}[0]{\dasharrow}
\newcommand{\qtq}[1]{\quad\mbox{#1}\quad}

\newcommand{\supp}[0]{\operatorname{Supp}}

\def\into{\DOTSB\lhook\joinrel\to}

\usepackage[all]{xy}\xyoption{dvips}

\newcommand{\s}[0]{{\mathbb S}}
\newcommand{\sph}[0]{{\mathbb S}^2}
\newcommand{\trs}[0]{{\mathbb T}^2}
\newcommand{\kk}[0]{{\mathbb K}^2}
\newcommand{\rp}[0]{{\mathbb R}{\mathbb P}^2}

\newcommand{\li}[0]{{\mathbb L}}
\newcommand{\fib}[0]{{\mathbb F}}

\begin{document}

\bibliographystyle{amsalpha}

\title[Approximating curves on real rational surfaces]{Approximating curves on real rational surfaces}

\author{J\'anos Koll\'ar \and Fr\'ed\'eric Mangolte}

\address{Princeton University, Princeton NJ 08544-1000}

\email{kollar@math.princeton.edu}

\address{LUNAM Universit\'e, LAREMA, CNRS-Universit\'e d'Angers,  49045 Angers, vFrance.}

\email{frederic.mangolte@univ-angers.fr}

\today

\begin{abstract}
We give necessary and sufficient topological conditions for
 a simple closed curve on a real rational surface to be approximable by
 smooth rational curves. We also study  approximation by
 smooth rational curves with given  complex  self-intersection
number.
\end{abstract}

\maketitle

\section{Introduction}

As a generalization of the Weierstrass approximation theorem,
every $C^{\infty}$ map to a rational variety $\s^1\to X$
can be approximated, in the  $C^{\infty}$-topology,  by  real algebraic maps
$\r\p^1\to X$; see \cite{bk3}
and Definition \ref{epproximate.defn}. In this article we study the
following variant of this result.

\begin{ques} Let $X$ be a smooth real algebraic variety and $L\subset X(\r)$
a smooth, simple, closed curve.  Can it be
approximated, in the  $C^{\infty}$-topology,  by the real points
of a smooth rational curve $C\subset X$?
\end{ques}

\begin{defn}[Real algebraic varieties]
For us a {\it real algebraic variety} is an algebraic variety,
as in \cite{shaf}, that is
defined over $\r$. If $X$ is a real algebraic variety
then $X(\c)$ denotes the set of complex points and
 $X(\r)$ the set of real points.
(Note that frequently -- for instance in
the  book  \cite{bcr} --
 $X(\r)$ itself is called a real algebraic variety.)
Thus for us $\p^n$ is a real algebraic variety whose real points
$\p^n(\r)$ can be identified with $\r\p^n$  and whose complex points
$\p^n(\c)$ can be identified with $\c\p^n$.

If $X\subset \p^n$ is a  quasi-projective real algebraic variety then
$X(\r)$ inherits from  $\r\p^n$ a (Euclidean) topology;
if $X$ is smooth, it inherits a differentiable structure.
In this article, we  always use this topology and differentiable structure.

For many purposes, the behavior of a real variety at its
complex points is not relevant, but in this paper it is crucial to
consider complex  points as well.
When we talk about a smooth, projective,  real  algebraic variety,
it is important that smoothness hold at all complex points
and $X(\c)$ be  compact.


We say that a real algebraic variety $X$ of dimension $n$ is
{\it rational} if it is birational to  $\p^n$; that is, the
birational map  is also defined over $\r$.
If such a birational map exists with complex coefficients,
we say that $X$ is {\it geometrically rational.}
\end{defn}

If $X$ is a  rational variety and $\dim X\geq 3$ then one can easily
perturb  the approximating maps $\p^1\to X$  produced by the
proof of \cite{bk3}
to obtain embeddings; see Proposition \ref{higher.dim.say}.1.
However, if $S$ is an algebraic surface,
 then usually there are very few embeddings
$\p^1\into S$; for instance only lines and conics for $S=\p^2$.

As a simple example, consider the parametrization of the nodal plane cubic curve
 $\bigl((x^2+y^2)z=x^3\bigr)$ given by
$f\colon (u,v)\longrightarrow \bigl(v(u^2+v^2),u(u^2+v^2),v^3\bigr)$.
Clearly $f(\r\p^1)$ is a simple closed curve in $\r\p^2$ but its Zariski
closure has an extra isolated real
point at $(0,0,1)$. One can remove this point either
by perturbing the equation to
$\bigl(z(x^2+y^2+\epsilon^2z^2)=x^3\bigr)$
or by  blowing up the point $(0,0,1)$. In the first case
  the curve becomes  elliptic, in the second case the
topology of the real surface changes.
In this paper we aim to  get rid of such extra real singular points.

By the above remarks,  the best one can hope for is to get approximation
by rational curves $C\subset S$ such that $C$ is smooth at its real points.
We call such curves {\it real-smooth.}
The main result is the following.

\begin{thm}\label{thm.main}
Let $S(\r)$ be the underlying topological surface of the real points of a
smooth rational surface and
$L\subset S(\r)$ a simple, connected, closed curve.
The following are equivalent.
\begin{enumerate}
\item  $L$ can be approximated by  real-smooth
 rational curves in the  $C^{\infty}$-topology.
\item There is a smooth rational surface  $S'$ and a  smooth
 rational curve $C'\subset S'$ such that  $(S,L)$ is diffeomorphic to
$(S'(\r), C'(\r))$.
\item  $(S(\r),L)$ is {\bf not} diffeomorphic to the pair
$(\mbox{torus}, \mbox{null homotopic curve})$.
\end{enumerate}
\end{thm}

As we noted, for a given $S$, our approximating curves
almost always have many singular points,
but they come in complex conjugate pairs. These singular points can be blown up
without changing the real part of $S$.
This shows that (\ref{thm.main}.1) $\Rightarrow$ (\ref{thm.main}.2)
and we explain later how
(\ref{thm.main}.2) $\Rightarrow$ (\ref{thm.main}.1)
can be  derived from the  results of \cite{bh, k-mang}.
The implication (\ref{thm.main}.2) $\Rightarrow$ (\ref{thm.main}.3)
turns out to be a straightforward genus computation in Proposition
\ref{torus.null-hom.lem}.
The main result is
(\ref{thm.main}.3) $\Rightarrow$ (\ref{thm.main}.2),
which is proved by enumerating all possible
topological pairs $(S(\r),L)$ and then exhibiting each
for a suitable rational surface, with one exception as in
(\ref{thm.main}.3).

In order to state a more precise version, we fix our topological
notation.

\begin{notation} Let $\s^1$  denote  the circle, $\sph$ the 2-sphere,
$\rp$ the real projective plane,
$\trs\sim \s^1\times \s^1$ the 2-torus and  $\kk$ the Klein bottle.

We also use some standard curves on these surfaces.
$\li\subset \sph$ denotes a circle and $\li\subset \rp$   a line.
We think of both $\trs$ and $\kk$  as an $\s^1$-bundle over $\s^1$.
Then $\li$ denotes  a section and  $\fib$ a fiber.
Note that $(\trs, \li)$ is diffeomorphic to $(\trs, \fib)$
but  $(\kk, \li)$ is not diffeomorphic to $(\kk, \fib)$.

Diffeomorphism of two surfaces $S_1, S_2$ is denoted by
 $S_1\sim S_2$.
Connected sum with $r$ copies of $\rp$ (resp.\ $\trs$)
is denoted by $\#r\rp$  (resp.\ $\#r\trs$).

Let $(S_1,L_1)$ be  a surface and a curve on it.
Its connected sum with a surface $S_2$ is denoted by
$(S_1,L_1)\# S_2$. Its underlying surface is $S_1\# S_2$.
We assume that the connected sum operation is disjoint from $L_1$;
then we get $L_1\subset S_1\# S_2$. This operation is well defined
if $S_1\setminus L_1$ is connected. If  $S_1\setminus L_1$ is disconnected,
then it matters to which side we attach $S_2$.
In the latter case we distinguish these by putting $\# S_2$
on the left or right of  $(S_1,L_1)$.
Thus
$$
 r_1\rp\#  (\sph, L)\# r_2\rp
$$
indicates that we attach $r_1$ copies of $\rp $
to one side of $\sph\setminus L$ and  $r_2$ copies of $\rp $
to the other side.

We  also need to take connected sums of the form
$(S_1,L_1)\#(S_2,L_2)$. From both surfaces we remove a disc
that intersects the curves $L_i$ in an interval; we can think of the
boundaries as $\s^1$ with 2  marked points  $(\s^1, \pm 1)$.
Then we glue so as to get a simple closed curve on  $S_1\# S_2$.
In general there are 4 ways of doing this, corresponding to the
4 isotopy classes of self-diffeomorphisms of $(\s^1, \pm 1)$. However, when
one of the pairs is $(\rp, \li)=\bigl(\rp, (x=0)\bigr)$
and we remove the disc  $(x^2+y^2\leq z^2)$,
then the automorphisms  $(x{:}y{:}z)\mapsto (\pm x{:}\pm y{:}z)$
represent all 4 isotopy classes, hence
the end result $(S_1,L_1)\#(\rp, \li)$ is unique.
This is the only case that we use.

\end{notation}

\begin{defn}[Intersection numbers]
 The intersection number of two algebraic  curves $\bigl(C_1\cdot C_2\bigr)$
on a smooth, projective surface
is the intersection number of the underlying complex curves.
The intersection number of the real parts
$\bigl(C_1(\r)\cdot C_2(\r)\bigr)$ is only defined modulo 2 and
$$
\bigl(C_1(\r)\cdot C_2(\r)\bigr)\equiv \bigl(C_1\cdot C_2\bigr) \mod 2.
$$
In particular, if $C\subset S$ is a rational curve such that
$C$ is smooth at its real points, then
$S(\r)$ is orientable along $C(\r)$  $\Leftrightarrow$
$\bigl(C(\r)\cdot C(\r)\bigr)\equiv 0 \mod 2$  $\Leftrightarrow$
$\bigl(C\cdot C\bigr)$ is even.
\end{defn}

The following result  lists the
topological pairs $\bigl(S(\r), C(\r)\bigr)$,
depending  on the complex
self-intersection of the rational curve.
In the table below we ignore the trivial cases when
$C(\r)=\emptyset$. We see in (\ref{mmp.steps}.2) that
  every topological type that occurs
 for $(C^2)=e+2$ also
occurs for $(C^2)=e$;  thus, for clarity,  line $e$ lists only  those types
that do not appear for $e+2,\dots, e+8$. We call these the
{\it new topological types.}

\begin{thm} \label{m.geq-2.thm}
Let $S$ be a smooth, projective surface defined and rational
over $\r$ and $C\subset S$  a rational curve that is smooth
(even over $\c$).
For $(C^2)\geq -2$ the following table lists the possible
topological types of the pair $\bigl(S(\r), C(\r)\bigr)$.
$$
\begin{array}{cl}
(C^2) & \mbox{new topological types}\\
\mbox{even}\geq 6 & (\trs, \li)\# r\rp:\quad r=0,1,\dots\\
\mbox{odd}\geq 5 & (\kk, \li)\# r\rp:\quad r=0,1,\dots\\
4 & r_1\rp\#  (\sph, \li)\# r_2\rp:\quad r_1+r_2\geq 1\\
3 & \mbox{nothing new}\\
2 & (\sph, \li)\\
1 & (\rp, \li)\\
0 & (\kk, \fib)\\
-1 & (\rp, \li)\# \trs\\
-2 &  (\kk, \fib)\# \trs.
\end{array}
$$
\end{thm}

Thus, as an example,   the possible
topological types of the pair $\bigl(S(\r), C(\r)\bigr)$
where $(C^2)=0$ are given by the entries corresponding to
the values $(C^2)=0,2,4$  and $\mbox{\it even}\geq 6$.

As we see in Section \ref{sec.mmp},  the entries for $e\geq -1$ follow
from an application of the minimal model program
to the pair $\bigl(S,(1-\epsilon)C\bigr)$. Nothing unexpected happens for
$e=-2$ but this depends on some rather delicate properties
of singular Del Pezzo surfaces;
see Lemma \ref{-2.case.deg2.dp}.

By contrast, we know very little about the cases $e\leq -3$.
These lead to the study of rational surfaces with quotient singularities
and ample canonical class. There are many such cases -- see
\cite[Sec.5]{k-bmy} or Example \ref{ample.K.exmp} --
but very few definitive results
\cite{MR2824965, MR2888175, MR2805643,MR2786664}.

The pairs $(S,L)$ listed in Theorem \ref{m.geq-2.thm}
and the pairs easily derivable from them give almost all
examples needed to prove (\ref{thm.main}.3) $\Rightarrow$ (\ref{thm.main}.2).
The only exceptions are pairs  $(S,L)$
where $S\setminus L$ is the disjoint union of a M\"obius strip
and of an orientable surface of genus $\geq 2$.
These are constructed by hand in Example \ref{final.exmp.series}.

We use the following  basic result on
the topology of real algebraic surfaces, due to \cite{Co14}.

\begin{thm}[Comessatti's theorem]\label{Co14.thm}
 Let $S$ be a projective, smooth real algebraic surface
that is birational to $\p^2$. Then $S(\r)$ is either
$\sph, \trs$ or $\#r\rp$ for some $r\geq 1$.\qed
\end{thm}

\begin{defn}\label{epproximate.defn}
For a differentiable manifold $M$, let $C^{\infty}(\s^1,M)$ denote the space
of all $C^{\infty}$ maps
 of $\s^1$ to $M$, endowed with the $C^{\infty}$-topology.

Let $X$ be a smooth real algebraic variety and
$C\subset X$ a  rational curve.
By choosing any isomorphism of its normalization $\bar C$ with the
plane conic  $(x^2+y^2=z^2)\subset \p^2$, we get a
$C^{\infty}$ map  $\s^1\to X(\r)$ whose image coincides with
$C(\r)$, aside from its isolated real singular points.

Let $\sigma: L\hookrightarrow X(\r)$ be an embedded circle.
We say that $L$ can be  \emph{approximated by}
 rational   curves of a certain kind if every
 neighborhood of $\sigma$ in $C^{\infty}\bigl(\s^1,X(\r)\bigr)$
contains a map derived as above from  a curve of that kind.
\end{defn}

\begin{ack} We thank V.~Kharlamov for useful
conversations and the referee for many very helpful
corrections and suggestions.
Partial financial support  for JK was provided by  the NSF under grant number
 DMS-07-58275.
The research of FM was partially supported by ANR Grant ``BirPol''
 ANR-11-JS01-004-01.
\end{ack}

\section{Minimal models for pairs}\label{sec.mmp}

 Let
${\mathcal S}$ be a class of smooth, projective surfaces defined
over $\r$ that is closed under birational equivalence.
We would like to understand the possible topological types
$\bigl(S(\r), C(\r)\bigr)$ where $S\in {\mathcal S}$
and $C\subset S$ is a smooth, rational curve.

We are mostly interested in the cases when ${\mathcal S}$
consists of rational or geometrically  rational surfaces.
It is not a priori obvious, but the answer turns out to have an interesting
dependence on the self-intersection number $e:=(C^2)$.

Our approach is to run the
$\bigl(S, K_S+(1-\epsilon)C\bigr)$-minimal model program
(abbreviated as MMP) for  $0<\epsilon\ll 1$;
see (\ref{mmp.steps}.2) why the $-\epsilon $ is needed.
 (For a general introduction to MMP
over any field, see  \cite[Sec.1.4]{km-book}.
The real case  is discussed for smooth surfaces in   \cite[Sec.2]{k-top}
and for  surfaces with Du~Val singularities in
\cite[Sec.2]{MR1760882}.)
Then we need to understand how the topology of $\bigl(S(\r), C(\r)\bigr)$
changes with the steps of  the program and describe the possible
last steps. At the end we try to work backwards to get our final answer.

Since $\bigl(C\cdot (K_S+C)\bigr)=-2$, the divisor $K_S+(1-\epsilon)C$
has negative intersection number with $C$ for  $0<\epsilon\ll 1$,
so the minimal model program always produces a nontrivial contraction
$\pi:S\to S_1$. If $\pi$ is birational and $C$ is not $\pi$-exceptional,
set $C_1:=\pi(C)\subset S_1$.

Note that if $E\subset S$ is an irreducible  curve such that
$\bigl(K_S+(1-\epsilon)C\bigr)\cdot E<0$ then $(K_S\cdot E)<0$,
except possibly when
$E=C$ and $(C^2)<0$. Thus -- aside from the latter case which we discuss in
(\ref{mmp.steps}.5) -- all steps of the
$\bigl(K_S+(1-\epsilon)C\bigr) $-MMP are also steps on the traditional MMP.

\begin{say}[List of the possible steps of the MMP]\label{mmp.steps}{  \ }

In what follows, we ignore the few cases where  $S(\r)=\emptyset$
since these are not relevant for us.

\medskip
{\it Elementary birational contractions.} Here $S_1$ is a smooth surface
and $\pi:S\to S_1$ is obtained by blowing up a real point or a
conjugate pair of complex points. There are 4 cases.
\medskip

(\ref{mmp.steps}.1) $\pi$ contracts a conjugate pair of
disjoint $(-1)$-curves
that are disjoint from $C$.
Then
$$
\bigl(S(\r), C(\r)\bigr)\sim \bigl(S_1(\r), C_1(\r)\bigr)
\qtq{and}(C^2)=(C_1^2).
$$

(\ref{mmp.steps}.2) $\pi$ contracts a conjugate pair of disjoint  $(-1)$-curves
that intersect $C$ with multiplicity 1 each.
(Note that  $\bigl(E\cdot \bigl(K_S+(1-\epsilon)C\bigr)\bigr)=-\epsilon<0$;
this is why we needed the $-\epsilon $ perturbation term.)
Then
$$
\bigl(S(\r), C(\r)\bigr)\sim\bigl(S_1(\r), C_1(\r)\bigr)\qtq{and}
(C^2)=(C_1^2)-2.
$$
 The inverse shows that  every
topological type that occurs for $(C^2)=e$ also
occurs for $(C^2)=e-2$.

(\ref{mmp.steps}.3) $\pi$ contracts a real $(-1)$-curve
that is disjoint from $C$.
Then
$$
\bigl(S(\r), C(\r)\bigr)\sim\bigl(S_1(\r), C_1(\r)\bigr)\#\rp\qtq{and}
(C_1^2)=(C^2).
$$
 The inverse shows that  for every
topological type that occurs, its connected sum with $\rp$
also occurs.

(\ref{mmp.steps}.4) $\pi$ contracts a real $(-1)$-curve
that intersects $C$ with multiplicity 1.
Then
$$
\bigl(S(\r), C(\r)\bigr)\sim\bigl(S_1(\r), C_1(\r)\bigr)\#(\rp,\li)\qtq{and}
(C^2)=(C_1^2)-1.
$$
\medskip

With these birational contractions,
$\bigl(S_1, C_1\bigr) $ is again a pair in our class
$\mathcal S$
and we can continue running the minimal model program
to get
$$
(S, C)=(S_0, C_0)\stackrel{\pi_0}{\to}(S_1, C_1)\stackrel{\pi_1}{\to}
\cdots \stackrel{\pi_{m-1} }{\to}(S_{m}, C_{m})
$$
until no such contractions are possible. We call such pairs
$(S_{m}, C_{m}) $ {\it classically minimal.}
Note also that  in any sequence of these steps,
the value of $(C_i^2)$ is non-decreasing.

\medskip
{\it Singular birational contraction.}
\medskip

(\ref{mmp.steps}.5) $\pi$ contracts $C$ to a point.
This can happen only if $(C^2)<0$. If $(C^2)\leq -2$, the resulting
$S_1$ is singular. For $(C^2)\leq -3$
these are very difficult cases and we try to avoid
them if possible.

\medskip
{\it Non-birational contractions.}
\medskip

(\ref{mmp.steps}.6) $\pi:S\to \p^1$ is a $\p^1$-bundle and
$C$ is a section. In this case
$$
\bigl(S(\r), C(\r)\bigr)\sim (\trs, \li)\qtq{or} (\kk, \li)
\qtq{with $(C^2)$ arbitrary.}
 $$
(More precisely, $(C^2)$ is even for $\trs$ and odd for $\kk$.)

(\ref{mmp.steps}.7) $\pi$ maps $S=\p^2$ to a point and $C$ is a conic. Thus
$$
\bigl(S(\r), C(\r)\bigr)\sim  (\sph, \li)\#\rp\qtq{and} (C^2)=4.
$$

(\ref{mmp.steps}.8)  $\pi$ maps $S=(x^2+y^2+z^2=t^2)\subset \p^3 $ to a point  and $C=(x=0)$
 is a plane section. Thus
$$
\bigl(S(\r), C(\r)\bigr)\sim (\sph, \li)\qtq{and} (C^2)=2.
$$
(Note that  the hyperboloid
$T:=(x^2+y^2=z^2+t^2)\subset \p^3 $ is isomorphic to
$\p^1\times \p^1$, and the corresponding step of the MMP is
either one of the coordinate projections.
This is listed under (\ref{mmp.steps}.6).)

(\ref{mmp.steps}.9)  $\pi$ maps $S=\p^2$ to a point  and $C$ is a line. Thus
$$
\bigl(S(\r), C(\r)\bigr)\sim (\rp, \li)\qtq{and} (C^2)=1.
$$

(\ref{mmp.steps}.10) $\pi:S\to \p^1$ is a conic bundle and
$C$ is a  smooth fiber. If $S$ is rational then we have
three possibilities
$$
\bigl(S(\r), C(\r)\bigr)\sim (\trs, \fib), (\kk, \fib) \qtq{or} (\sph, \li)
\qtq{ and $(C^2)=0$.}
$$

(Note: If $S$ is geometrically rational but not necessarily rational,
then Steps \ref{mmp.steps}.1--9  are unchanged, but in Step \ref{mmp.steps}.10
 we can also have
the disjoint union of $(\sph, \li) $ with  copies of $\sph$.)
\end{say}

Putting these together, we get the following.

\begin{cor}\label{mmp.top.backward.cor}
  Let $S$ be a smooth, projective, rational  surface defined
over $\r$ and $C\subset S$  a smooth, rational curve.
Run the $\bigl(K_S+(1-\epsilon)C\bigr)$-MMP to get
$$
(S, C)=(S_0, C_0)\stackrel{\pi_0}{\to}(S_1, C_1)\stackrel{\pi_1}{\to}
\cdots \stackrel{\pi_{m-1} }{\to}(S_{m}, C_{m})
\stackrel{\tau }{\to} T.
$$
Assume that the $\pi_i$ are elementary contractions as in
(\ref{mmp.steps}.1--4) and $\tau$ is a non-birational contraction
as in (\ref{mmp.steps}.6--10).

Then $S_m, C_m$ are smooth and
$\bigl(S(\r), C(\r)\bigr)$ can be described as follows.
\begin{enumerate}
\item  If $S(\r)\setminus C(\r) $ is connected then
$$
\bigl(S(\r), C(\r)\bigr)\sim
 \bigl(S_m(\r), C_m(\r)\bigr) \# r_1(\rp, \li)\# r_2\rp
$$
for some $r_1, r_2\geq 0$. Here $(C^2)\leq (C_m^2)-r_1$,
with strict inequality if we ever perform Step \ref{mmp.steps}.2.
\item If $S(\r)\setminus C(\r) $ is disconnected then
 $\bigl(S_m(\r), C_m(\r)\bigr)$ is $(\sph, \li)$
or $(\sph, \li)\#\rp$  and all  real exceptional curves
of $S\to S_m$ are disjoint from $C(\r)$.
In this case $C_m(\r)$ separates $S_m(\r)$ and in taking connected sums
we need to keep track on which side we blow up. Thus
$$
\bigl(S(\r), C(\r)\bigr)\sim
r_1\rp\# \bigl(S_m(\r), C_m(\r)\bigr)\# r_2\rp
$$
for some $r_1, r_2\geq 0$. As before,  $(C^2)\leq (C_m^2)$. \qed
\end{enumerate}
\end{cor}

It remains to understand what happens if the MMP ends with a
singular birational contraction. We start with the  simplest,
 $(C_m^2)=-1$ case; here the adjective ``singular'' is not warranted.

\begin{say}[Case $(C_m^2)=-1$] \label{-1.case.say}
Let
$C_m\subset S_m$ be a $(-1)$-curve and
 $\tau:S_m\to S_{m+1}$ its contraction. Then
 $S_{m+1}$ is again a surface in
${\mathcal S}$.  Thus
$$
\bigl(S_m(\r), C_m(\r)\bigr)\sim (\rp, \li)\# S_{m+1}(\r).
$$
If $S$ is a rational surface then so is $S_{m+1}$
thus, by Theorem~\ref{Co14.thm}, we have only the cases
$$
\bigl(S_m(\r), C_m(\r)\bigr)\sim (\rp, \li)\#r\rp\qtq{or} (\rp, \li)\#\trs.
\eqno{(\ref{-1.case.say}.1)}
$$
We can then obtain $\bigl(S(\r), C(\r)\bigr) $
from $\bigl(S_m(\r), C_m(\r)\bigr) $ as in
Corollary \ref{mmp.top.backward.cor}.1.
\end{say}


\begin{rem}
Although we do not need it, note that if
$(S_m, C_m)$ is classically minimal then  every
$(-1)$-curve on $S_{m+1}$ passes through $\tau(C_m)$.
The latter condition is not sufficient to ensure
that $S_m$ be classically minimal, but it is easy to write down
series of examples.

Start with $P_0=\p^2$, a line $L\subset P_0$ and a point $p\in L$.
Blow up $p$ repeatedly to obtain $P_r$ with $C_r\subset P_r$
the last exceptional curve. We claim that
$C_r$ is the only $(-1)$-curve on $P_r$ for $r\geq 3$, thus $(P_r, C_r)$ is
classically minimal.

We can fix coordinates on $P_0$ such that $L=(y=0)$ and $p=(0:0:1)$.
Then the $(\c^*)^2$-action
$(x:y:z)\mapsto (\lambda x: \mu y: z)$ lifts to $P_r$,
hence the only possible curves with negative self-intersection on $P_r$
are the preimages of the coordinate axes and the exceptional curves of
$P_r\to P_0$. These are easy to compute explicitly.
Their dual graph is a cycle of rational curves
$$
  \xymatrix{
    (-2)\ar@{-}[d]\ar@{-}[r]& (-2)\ar@{--}[r] &
(-2) \ar@{-}[r]  & (-2)\ar@{-}[d]  \\
    (-1) \ar@{-}[r] & (1-r) \ar@{-}[r] & (1) \ar@{-}[r] & (0)
 }
  $$
where  $(a)$ denotes a curve of self-intersection $a$,
each curve intersects only the two neighbors connected to it by
a solid line and
there are $r-1$ curves with self-intersection $-2$ in the top row.
Thus
$C_r$ is the only $(-1)$-curve  for $r\geq 3$.

There are probably many more series of such surfaces.
\end{rem}

Next we study singular birational contractions where $(C_m^2)=-2$.
To simplify notation, we drop the subscript $m$.
The result and the proof remain the same over
an arbitrary field of characteristic 0. In this setting, a
pair $(S,C)$ is
{\it classically minimal} if there is no birational contraction
that is extremal both for $K_S$ and for $K_S+(1-\epsilon)C$.

\begin{lem} \label{-2.case.deg2.dp}
Let $S$ be a smooth,  geometrically rational surface
(over an arbitrary field $k$ of characteristic 0) and $C\subset S$
a smooth, geometrically rational curve. Assume that the pair $(S,C)$ is
classically minimal and
$(C^2)=-2$. Let $\pi\colon S\to T$ be the contraction of $C$.
Then $T$ is a singular  Del Pezzo surface with  Picard number 1 over $k$ and
one of the following holds.
\begin{enumerate}
\item $T$ is a quadric cone, hence
$S$ is a $\p^1$-bundle over a smooth, rational curve
  and $C\subset S$ is a section.
\item $T$ is a  degree 1 Del Pezzo surface. Furthermore,
there is a smooth, degree 2 Del Pezzo surface $S_1$
with Picard number 1 and a rational curve $C_1\in |-K_{S_1}|$
with a unique singular point $p_1\in C_1$ such that
$S= B_{p_1}S_1$ and $C$ is the birational transform of $C_1$.
\item  $T$ is a  degree 2 Del Pezzo surface. Furthermore,
there is a conic bundle structure
$\rho:S\to B$ whose fibers are the curves in
$|-K_T|$ that pass through the singular point.
The curve $C$ is a double section of $\rho$.
\end{enumerate}
In the last 2 cases, $S$ is not rational.
\end{lem}

\begin{proof}
Let $\pi\colon S\to T$ be the contraction of $C$.
Then $T$ has an ordinary node $q\in T$.
The special feature of the $(C^2)=-2$ case is that
$K_S\sim \pi^*K_T$, thus $K_T$ is not nef since $S$ is a
smooth rational surface.
 So there is an extremal
contraction  $\tau\colon T\to T_1$.

There are 3 possibilities for $\tau$.

Case 1: $\tau$ is birational with exceptional curve $E\subset T$.
Note that, over $\bar k$, $E$ is the disjoint union of $(-1)$-curves
that are conjugate to each other over $k$.

If $q$ does not lie on $E$ then $E$ gives a  disjoint union of $(-1)$-curves
 on $S$ which is
disjoint from $C$, a contradiction to the classical minimality assumption.
 If $q$ lies on $E$ then
$E$ is geometrically irreducible and $T_1$ is smooth
since  on a surface with Du~Val singularities, every extremal contraction
results in a smooth point, cf.\ \cite[Thm.2.6.3]{MR1760882}.

Thus the composite
$\tau\circ \pi\colon S\to T_1$ consist of two smooth blow ups.
This again shows that $(S,C)$ is not classically minimal.

Case 2: $\tau\colon T\to T_1$ is a conic bundle. Then
$\tau\circ \pi\colon S\to T_1$ is a non-minimal conic bundle,
hence there is a $(-1)$-curve $E$ contained in a fiber.
$C$ is also contained in a fiber thus $(E\cdot C)\leq 1$
since any 2 irreducible curves in a fiber of a
conic bundle intersect in at most 1 point. Thus again $(S, C)$ is
not classically minimal.

Case 3: $T$ is a Del Pezzo surface of Picard number 1 over $k$.

Since $K_S\sim \pi^*K_T$,
in this case $S$ itself is a weak Del Pezzo surface
(that is $-K_S$ is nef) of Picard number 2.
 Thus $S$ has another extremal ray giving a contraction
$\rho\colon S\to S_1$.
Next we study the possible types of $\rho$.

We use that for  every Del Pezzo surface $X$, the linear system
 $|-K_X|$ has dimension $\geq 1$.
A general member of $|-K_X|$ is smooth, elliptic;
special members are either irreducible, rational with a single
node or cusp or reducible with smooth, rational geometric components.

Case 3.1: $\rho$ is a $\p^1$-bundle. Then $C$ has to be the unique
 negative section,
giving the first possibility.

Case 3.2:  $\rho$ is birational so
$S_1$ is a Del Pezzo surface of Picard number 1.
Since $(S,C)$ is classically minimal, the exceptional curve of $\rho$
has intersection number $\geq 2$ with $C$. In particular
$C_1:=\rho(C)$ is singular.

Since  $|-K_S|$ has dimension $\geq 1$,
 there is a  divisor
$D\in |-K_S|$ such that $(C\cap D)\neq \emptyset$.
 On the other hand
$(C\cdot D)=(C\cdot K_S)=0$, hence $C\subset \supp D$.

Thus $C_1:=\rho(C)$ is singular and is contained in a member of $|-K_{S_1}|$.
Thus $C_1$ is a member of $|-K_{S_1}|$ and has a node or cusp at a point $p_1$.

From $-2=(C^2)=(C_1^2)-4$
we see that
 $S_1$ is a smooth Del Pezzo surface of degree 2.
We obtain $S$ by blowing up the singular point of $C_1$
and so $(K_T^2)=(K_S^2)=(K_{S_1}^2)-1=1$; giving the second possibility.

Case 3.3.  $\rho$ is a minimal conic bundle, that is, the Picard group of $S$
is generated by $K_S$ and  a general fiber $F\subset S$  of $\rho$.
Thus $C\sim aK_S+bF$ for some $a,b\in \z$. If $(K_S^2)=d$ this gives that
$$
-2= (C^2)=a^2(K_S^2)+2ab(K_S\cdot F)=a^2d-4ab=a(ad-4b).
$$
Since $C$ is effective, $a\leq 0$, hence we see that $a=-1$ and
using that $1\leq d\leq 9$ we obtain that
either $d=2, b=-1$ or $d=6, b=-2$. In the latter, the adjunction
formula gives  $C(C+K_S)=-4$, hence $C$ is reducible. Thus
$d=2$,  giving the third possibility.

Finally note that a Del Pezzo surface of degree 2 and of
Picard number 1
or a conic bundle of degree 2 and of
Picard number 2
is never rational over the ground field $k$
by the Segre--Manin theorem; see \cite{segre51, M66}
or \cite[Chap.2]{ksc} for an introduction to these results.
\end{proof}

The main difficulty with the $(C^2)\leq -3$ cases is that
contracting such a curve can yield a rational surface
with trivial or ample canonical class.
Here are some simple examples of this. For $d=6$ the example
below has $(C^2)=-4$; we do not know such pairs with
$(C^2)=-3$.

\begin{exmp} \label{ample.K.exmp}
Let $\bar C_d\subset \p^2$ be a rational curve of degree $d$
whose singularities are nodes. Thus we have  $\binom{d-1}{2}$ nodes
forming a set $N_d$. Let $p_d\colon S_d:=B_d\p^2\to \p^2$ denote the blow-up of all
the nodes with exceptional curves $E_d$
and $ C_d\subset  S_d$ the birational (or strict) transform
of $\bar C_d$. We compute that
$\bigl( C_d^2\bigr)=d^2-4\binom{d-1}{2}$
and
$$
K_{S_d}+\tfrac3{d}C_d-\bigl(1-\tfrac6{d}\bigr)E_d\sim_{\q}
p_d^*\bigl(K_{\p^2}+\tfrac3{d}\bar C_d\bigr)\sim_{\q} 0.
$$

If $d\geq 6$ then $\bigl( C_d^2\bigr)<0 $; let $\pi\colon S_d\to T_d$
be its contraction. Then
$$
K_{T_d}\sim_{\q} \bigl(1-\tfrac6{d}\bigr)\pi_*E_d
$$
is trivial for $d=6$ and ample for $d\geq 7$.
For $d=6$ this is a Coble surface \cite{dol-zha}.

\end{exmp}

\section{Topology of pairs  $(S,L)$}\label{sec.top}

In this section let $S$ denote the real part of a smooth, projective,
real algebraic surface that is rational over $\r$.
 By Theorem~\ref{Co14.thm}, $S$ is either
$\sph, \trs$ or $\# r\rp$ for some $r\geq 1$.
Let $L\subset S$ be a connected, simple, closed curve.
We aim to classify the pairs $(S, L)$ up to
diffeomorphism.
We distinguish 4 main cases.

\begin{say}[$S$ is orientable] \label{list.2}
Thus $S\sim\sph$ or $S\sim\trs$.  There are three possibilities
\begin{enumerate}
  \item  $(S,L)\sim(\sph, \li)$,
\item $(S,L)\sim(\trs, \li)$ and
\item $(S,L)\sim(\trs, \mbox{null-homotopic curve})$.
\end{enumerate}
\end{say}

\begin{say}[$S$ is not orientable along $L$] \label{list.1}
A neighborhood of $L$ is a M\"obius band and contracting $L$ we get another
topological surface $S'$ thus
$(S,L)\sim(\rp, \li)\# S'$. This gives two possibilities
\begin{enumerate}
  \item  $(S,L)\sim(\rp, \li) \# r\rp$ for some $r\geq 0$ or
\item $(S,L)\sim(\rp, \li) \# g\trs$ for some $g> 0$.
\end{enumerate}
\end{say}

In the remaining 2 cases $S$ is non-orientable but   orientable along $L$.

\begin{say}[$L$ is non-separating] \label{list.3}
Then we have another simple closed curve $L'\subset S$
such that $S$ is non-orientable along $L'$
and $L$ meets $L'$ at a single point transversally.
Then a neighborhood of $L\cup L'$  is a punctured Klein bottle
and  $(S,L)$ is the connected sum of $(\kk, \fib)$ with another surface.
 This gives two possibilities
\begin{enumerate}
  \item  $(S,L)\sim(\kk, \fib) \# r\rp$ for some $r\geq 0$ or
\item $(S,L)\sim(\kk, \fib) \# g\trs$ for some $g> 0$.
\end{enumerate}
\end{say}

\begin{say}[$L$ is separating] \label{list.4}
Then $S\setminus L$ has 2 connected components
and  at least one of them is non-orientable.
 This gives two possibilities
\begin{enumerate}
  \item  $(S,L)\sim r_1\rp\#   (\sph, \li)\#r_2\rp$ for some
$r_1+r_2\geq 1$ or
\item $(S,L)\sim r_1\rp\#   (\sph, \li)\# g\trs$ for some
$r_1,g> 0$.
\end{enumerate}
Since we can always create a connected sum with $\rp$ by blowing up a point,
for construction purposes the only new case that matters is
\begin{enumerate}\setcounter{enumi}{2}
\item $(S,L)\sim   \rp\#   (\sph, \li)\# g\trs $ for some  $g> 0$.
\end{enumerate}
\end{say}

By the formula (\ref{mmp.top.backward.cor}.1),
we need to understand connected sum with
$(\rp, \li)$. This is again easy, but usually not treated
in topology textbooks, so we state the formulas for ease of reference.

\begin{say}[Some diffeomorphisms]\label{some.diffeos}
 We start with the list of elementary steps.
$$
\begin{array}{lcl}
(\rp, \li)\# \rp & \sim & (\kk, \li)\\
(\trs, \li)\# \rp & \sim & (\kk, \li)\# \rp\\
(\trs, \li)\# (\rp, \li) & \sim & (\kk, \li)\# \rp\\
(\kk, \li)\# (\rp, \li) & \sim & (\trs, \li)\#\rp\\
(\sph, \li) \# (\rp, \li) & \sim & (\rp, \li)\\
(\rp, \li)\# (\rp, \li) & \sim &(\kk, \fib)\\
(\kk, \fib)\# (\rp, \li) & \sim & (\rp, \li)\# \trs\\
\end{array}
\eqno{(\ref{some.diffeos}.1)}
$$
There are -- probably many -- elementary topological ways to see these.
An approach using algebraic geometry is the following.

For the first, blow up a point in $\p^2$ not on the line $\li$.
We get a  minimal ruled surface over $\p^1$ and the line becomes a  section.

For the next three, take a minimal ruled surface $S$ over $\p^1$
with negative section $E$. If $(E^2)$ is even then
$\bigl(S(\r), E(\r)\bigr)\sim (\trs, \li)$ and if  $(E^2)$ is odd then
$\bigl(S(\r), E(\r)\bigr)\sim (\kk, \li)$. Blowing up a point on $E$ changes the
parity of $(E^2)$. Also, the fiber through that point becomes a
$(-1)$-curve $F'$ disjoint from the birational transform $E'$ of $E$.
We can contract $F'$ to get a
minimal ruled surface $S'$ over $\p^1$.

Blowing up a point  $p\in \li\subset \sph$ we get
$(\sph, \li) \# (\rp, \li)$. The conjugate lines through
$p$ become conjugate $(-1)$-curves and contracting them gives
$(\rp, \li) $.

Blowing up a point  $p\in \li\subset \rp$ we get a
 minimal ruled surface $S$ over $\p^1$. The exceptional curve $E$
is the negative section  and the birational transform $\li'$
of $\li$ is a fiber; this is $ (\kk, \fib)$.

Blowing up a point  $p\in \fib\subset \kk$, the
birational transform $\fib'$
of $\fib$ is a  $(-1)$-curve. As discussed at the beginning, contracting it
we get $\trs$, giving the last diffeomorphism.

Iterating these, we get the following list.
$$
\begin{array}{lcl}
(\trs, \li)\# 2r(\rp, \li) & \sim & (\trs, \li)\# 2r\rp\\
(\trs, \li)\# (2r+1)(\rp, \li) & \sim & (\kk, \li)\# (2r+1)\rp\\
(\kk, \li)\# 2r(\rp, \li) & \sim & (\kk, \li)\# 2r\rp\\
(\kk, \li)\# (2r+1)(\rp, \li) & \sim & (\trs, \li)\# (2r+1)\rp\\
(\sph, \li) \# 2r(\rp, \li) & \sim &(\kk, \fib)\# (r-1)\trs\quad (r\geq 1)\\
(\sph, \li) \# (2r+1)(\rp, \li) & \sim &(\rp, \li)\# r\trs\\
(\rp, \li) \# 2r(\rp, \li) & \sim &(\rp, \li)\# r\trs\\
(\rp, \li) \# (2r+1)(\rp, \li) & \sim &(\kk, \fib)\# r\trs\\
(\kk, \fib)\# 2r(\rp, \li) & \sim & (\kk, \fib)\# r\trs\\
(\kk, \fib)\#(2r+1)(\rp, \li) & \sim &(\rp, \li)\# (r+1)\trs
\end{array}
\eqno{(\ref{some.diffeos}.2)}
$$

\end{say}

\section{Proofs of the Theorems}\label{sec.pfs}

\begin{say}[Proof of Theorem \ref{m.geq-2.thm}]
Let $S$ be a smooth, projective real algebraic surface
 over $\r$ and $C\subset S$ a smooth rational curve.
We run the $\bigl(K_S+(1-\epsilon)C\bigr)$-MMP while we can perform
elementary contractions to get
$$
(S, C)=(S_0, C_0)\stackrel{\pi_0}{\to}(S_1, C_1)\stackrel{\pi_1}{\to}
\cdots \stackrel{\pi_{m-1} }{\to}(S_{m}, C_{m}).
$$
We saw that $(C_m^2)\geq (C^2)$.
If $S$ is rational (or geometrically rational) there is at least one more
 step of the $\bigl(K_S+(1-\epsilon)C\bigr)$-MMP
$$
\tau: (S_{m}, C_{m})\to T.
$$
If $(C_m^2)\geq 0$ then $\tau$ is a non-birational contraction as in
(\ref{mmp.steps}.6--10). The topology of
$\bigl(S_m(\r), C_m(\r)\bigr)$ is fully understood
and   Corollary \ref{mmp.top.backward.cor} shows how to get
$\bigl(S(\r), C(\r)\bigr)$ from $\bigl(S_m(\r), C_m(\r)\bigr)$.

In order to get all possible $\bigl(S(\r), C(\r)\bigr)$
with $(C^2)=e$ we proceed in 4 steps.
\begin{enumerate}
\item Describe all $\bigl(S_m(\r), C_m(\r)\bigr)$ with
$(C_m^2)\geq e$.
\item For any $0\leq r_1\leq (C^2)-(C_m^2)$ with $(C^2)-(C_m^2)-r_1$ even,
determine the topological types of
$\bigl(S_m(\r), C_m(\r)\bigr)\# r_1(\rp, \li)$, using the formulas (\ref{some.diffeos}.2).
\item For any of the surfaces  obtained in (2), determine the topological types
obtained by taking connected sum with any number of copies of $\rp$.
\item In order to get the new types in Theorem \ref{m.geq-2.thm},
for any $e$ remove those that also occur for
$e+2$.
\end{enumerate}
The   first seven lines of the table in Theorem \ref{m.geq-2.thm}
follow from these. The first 2 lines derive from the cases in
(\ref{mmp.steps}.6) the next 5 lines from
 the cases in
(\ref{mmp.steps}.7--10).

If $(C_m^2)=-1$ then we use
(\ref{-1.case.say}.1) and Corollary \ref{mmp.top.backward.cor}.

Finally, if $(C_m^2)=-2$ then
$\bigl(S_m(\r), C_m(\r)\bigr)\sim (\trs, \li)$
by Lemma \ref{-2.case.deg2.dp}, but this is already listed
in the first line.
The only new example comes from  $(S_m, C_m)=(\sph, \li) $
(corresponding to $(C_m^2)=2$) and 4 blow-ups on $C_m$:
$$
\bigl(S(\r), C(\r)\bigr)
\sim (\sph, \li)\# 4 (\rp, \li)\sim (\kk,\fib)\# \trs. \qed
$$

\end{say}

\begin{say}[Proof of Theorem \ref{thm.main}]
We already noted
  that (\ref{thm.main}.1) $\Rightarrow$ (\ref{thm.main}.2) is clear.

The converse,
(\ref{thm.main}.2) $\Rightarrow$ (\ref{thm.main}.1)
involves two steps. First, if $S_1, S_2$ are
 smooth, projective real algebraic surfaces
that are rational over $\r$ and $S_1(\r)\sim S_2(\r)$
then
 there is a birational map
$g:S_1\map S_2$ that is an isomorphism between suitable
Zariski open neighborhoods of $S_1(\r)$ and $ S_2(\r)$.
This is \cite[Thm.1.2]{bh}; see also \cite{hm3} for a more direct proof.

Thus we have  $L\subset S(\r)$ and a rational curve $C\subset S$
that is smooth at its real points and a diffeomorphism
$$
\phi: \bigl(S(\r), L\bigr)\sim \bigl(S(\r), C(\r)\bigr).
$$
By \cite{k-mang}, the diffeomorphism $\phi^{-1}$ can be
approximated in the $C^{\infty}$-topology by birational maps
$\psi_n:S\map S$  that are  isomorphisms between suitable
Zariski open neighborhoods of $S(\r)$.
Thus
$$
C_n:=\psi_n(C)\subset S
$$
is a sequence of real-smooth rational curves
and  $C_n(\r)\to L$ in the $C^{\infty}$-topology.
One can   resolve the complex singular points of $C_n$ to get
 approximation of $L$ by smooth rational curves
$\bigl(C'_n\subset S_n)$. Here the surfaces $S_n$ are
isomorphic near their real points but not  everywhere.

Again using \cite[Thm.1.2]{bh}, in order to show
(\ref{thm.main}.2) $\Rightarrow$ (\ref{thm.main}.3),
it is enough to prove that on $\p^1\times \p^1$
there are  no real-smooth  rational curves $C$ defined over $\r$
such that
$C(\r)$ is null-homotopic. This follows from a  genus computation
done in Proposition
\ref{torus.null-hom.lem}.

It remains to show that
(\ref{thm.main}.3) $\Rightarrow$ (\ref{thm.main}.2).
All possible
topological pairs $\bigl(S(\r),L\bigr)$
were  enumerated in (\ref{list.2}--\ref{list.4}).
With the exception of  cases
(\ref{list.2}.3) and (\ref{list.4}.1--2), the examples listed in
Theorem \ref{m.geq-2.thm},
and their descendants using  the formulas (\ref{some.diffeos}.2),
cover everything.
We already proved that (\ref{list.2}.3)  never occurs. This leaves us with
the task of exhibiting examples for (\ref{list.4}.1--2).
As noted there, we only need to find examples for
(\ref{list.4}.3); these are constructed next. \qed

\begin{exmp}\label{final.exmp.series}
Let $L_1,\dots, L_{g+1}$ be distinct lines through the origin in $\r^2$
and  $H(x,y)$ the equation of their union.
For some $0<\epsilon\ll 1$ let $\bar C^{\pm}\subset \p^2$
be the Zariski closure of the
image of the unit circle  $(x^2+y^2=1)$ under the map
$$
(x,y)\mapsto \bigl(1\pm\epsilon H(x,y)\bigr)(x,y).
$$
The curves  $\bar C^{\pm}$ are rational and intersect each other at the
$2g+2$ points where the unit circle intersects one of the lines $L_i$
and also at the conjugate
point pair  $(1:\pm i:0)$.  Note further
 that  $(1:\pm i:0)$ are the only points of
 $\bar C^{\pm}$ at infinity.


It is better
to use the inverse of the stereographic projection from the south pole
to compactify $\r^2_{xy}$  as the quadric
$Q^2:=(z_1^2+z_2^2+z_3^2=z_0^2)\subset \p^3$.
From $\p^2$ this is obtained by  blowing up the conjugate
point pair  $(1:\pm i:0)$ and contracting the birational
transform of the line at infinity.
We think of the image of the unit circle as the equator.
Thus we get  rational curves $C^{\pm}\subset Q^2 $.
Since $(1:\pm i:0)$ are the only points of
 $\bar C^{\pm}$ at infinity,
 the south pole is not on  the curves $C^{\pm}$
and so the real points of the curves $C^{\pm} $ are all smooth
and they  intersect each other at
$2g+2$ points on the equator.

Pick one of these points $p$ and view $C_0:=C^+\cup C^-$
as the image of a map  $\phi_0$ from the reducible curve
$B_0:=(uv=0)\subset \p^2_{uvw}$ to $Q^2$ that sends the point
$(0:0:1)$ to $p$. By \cite[Appl.17]{ar-ko} or \cite[II.7.6.1]{rc-book},
 $\phi_0$ can be deformed
to morphisms
$$
\phi_{t}: B_t:=(uv=tw^2)\to Q^2.
$$
Let $C_t\subset Q^2$ denote the image of $B_t$.
For $t$ near the origin and with suitable sign,
$C_t(\r)\subset \sph=Q(\r)$ goes around the equator twice
and has $2g+1$ self intersections; see Figure 1.

\begin{figure}
\begin{center}
\includegraphics[width=100mm]{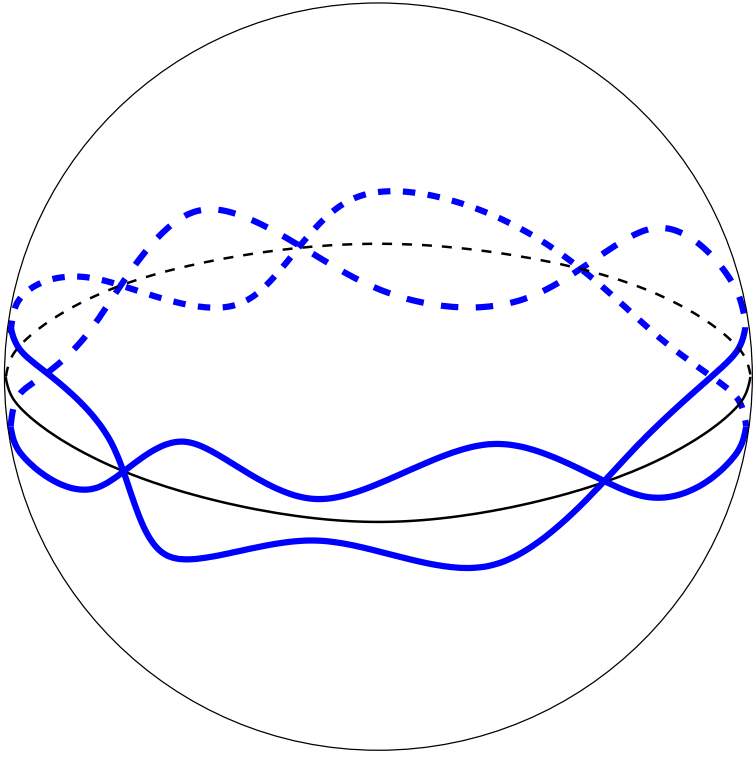}
\caption{The $g=2$ case}
\end{center}
\end{figure}

Finally we blow up the $2g+1$ real singular points of $C_t$
to get a rational surface  $S_g$. The birational transform of
$C_t$ gives a rational curve $C_g\subset S_g$ which is
smooth at its real points.

The  $2g+1$ regions of $\sph\setminus C_t(\r)$
near the equator become a M\"obius band on $S_g(\r)\setminus C_g(\r)$
and the northern and southern hemispheres
become  $\#g\trs$ (with one puncture).
Thus
$$
\bigl(S_g(\r), C_g(\r)\bigr)\sim \rp \# (\sph, \li)\#g\trs.
$$
\end{exmp}




\end{say}

\begin{prop}\label{torus.null-hom.lem}
Let  $C\subset \p^1\times \p^1$ be a  real-smooth
rational curve defined over $\r$.
 Then $[C(\r)]\in H_1(\trs, \z/2)$ is nonzero.
\end{prop}

Proof. Let $E_1, E_2$ denote  a horizontal (resp.\ vertical)
complex  line on $\p^1\times \p^1$.
Every complex algebraic curve $C$ has homology class $a_1E_1+a_2E_2$ for some
$a_1, a_2\geq 0$.  Furthermore, if $C$ is defined over $\r$ then
$$
a_i=(C\cdot E_{3-i})\equiv \bigl(C(\r)\cdot E_{3-i}(\r)\bigr)  \mod 2.
$$
Thus if $[C(\r)]\in H_1(\trs, \z/2)$ is zero  then $a_1, a_2$ are even.
By the adjunction formula
$$
2p_a(C)-2=\bigl(  a_1E_1+a_2E_2\bigr)\cdot\bigl(  (a_1-2)E_1+(a_2-2)E_2\bigr)
=a_1(a_2-2)+a_2(a_1-2),
$$
hence $p_a(C)=(a_1-1)(a_2-1)$.
Thus, if the $a_i$ are even then
 $p_a(C) $ is odd.
Therefore,  if $C$ is rational then  it has an odd number of singular points
and at least one of them has to be real.\qed






\section{Related approximation problems}\label{sect.probs}

\begin{say}[Approximation of curves on algebraic  surfaces]
When $S$ is a non-rational surface, we can ask for
several possible analogs of
 Theorem~\ref{thm.main}.

On many surfaces there are no rational curves at all,
thus the best one can hope for is approximation by higher genus curves.
Even for this, there are several well known obstructions.

First of all,  given a real algebraic surface $X$, a necessary condition for
 a smooth curve $C$ to admit an approximation by an algebraic curve is that
 its fundamental class $[C]$ belong to the group of algebraic cycles
 $\operatorname{H}^1_{\mathrm{alg}}(X,\z/2)$. The latter group is generally
 a proper subgroup of the cohomology group $\operatorname{H}^1(X,\z/2)$.
See \cite{bm61}  and  \cite[Sec.12.4]{bcr} for details.

The structure of these groups for various real algebraic surfaces of special
 type is computed in
\cite{ma94, ma97, mvh98, ma00, ma03}. These
 papers contain  the classification of totally algebraic
surfaces, that is surfaces such that
$\operatorname{H}^1_{\mathrm{alg}}(S,\z/2)=\operatorname{H}^1(S,\z/2)$,
among K3, Enriques, bi-elliptic, and properly elliptic surfaces.
In particular,
if $S$ is a non-orientable surface underlying an Enriques surface or a
bi-elliptic surface, then there are simple, connected, closed  curves
 on $S$ with no approximation by any algebraic
curve, see \cite[Thm.1.1]{mvh98} and \cite[Thm.0.1]{ma03}.

If $S$ is orientable, there can be further obstructions
involving $\operatorname{H}^1(S,\z)$.
For instance, let $S\subset \r\p^3$ be a very general  K3 surface.
By the Noether--Lefschetz theorem, the Picard group of $S(\c)$ is generated by
the hyperplane class. If $S$ is contained in $\r^3$ then
 the restriction of $\o_{\p^3}(1)$
to $S$ is trivial, thus only null-homotopic curves can be
approximated by algebraic curves.

Note also that if $S$ is a real K3 surface, then
by \cite{ma97}, there is a totally algebraic real K3 surface real deformation
 equivalent to $S$ (at least if $S$ is a non-maximal surface) thus in
general there is no purely topological obstruction to approximability for
 real K3 surfaces.
\end{say}

\begin{say}[Approximation of curves on geometrically rational   surfaces]
Geometrically rational  surface contain many rational curves,
so  approximation by real-smooth rational curves could be possible.
Any geometrically rational surface is totally algebraic but
there are not enough automorphisms to approximate all
diffeomorphisms, at least if the
number of connected components is greater than 2; see~\cite{bm1}.

Another  obstruction arises from the genus formula.
 For example, let $S$ be a degree~$2$ Del Pezzo surface with Picard number
$\rho(S)=1$ and $C\subset S$ a curve on it. Then $C\sim -aK_S$ for some
positive integer $a$ and so $C(C+K_S)=2a(a-1)$ is divisible by $4$. Thus
 the arithmetic genus $p_a(C)$ is odd hence every real rational curve on
 $S$ has an odd number of singular
points on $S(\c)$. These can not all be complex conjugate, thus
every  rational curve on $S$  has a real singular point.

It seems, however, that this type of parity obstruction for
approximation
does not occur on any other geometrically rational surface.
We hazard the hope that if $S$ is a geometrically rational surface
then every simple, connected, closed curve
can be approximated by real-smooth rational curves,
save when
either $S\sim \trs$ or
$S$ is  isomorphic to a degree~$2$ Del Pezzo surface with Picard number $1$.

As another  generalization, one can study such problems for
singular  rational surfaces as in  \cite{hm4}. See also the series
\cite{cm1}, \cite{cm2} for the classification of
geometrically rational surfaces with Du~Val  singularities.
\end{say}

\begin{say}[Approximation of curves on higher dimensional varieties]
\label{higher.dim.say}
As for surfaces, we can hope to approximate every
simple, connected, closed curve on a real variety $X$
by a nonsingular rational curve
over $\r$ only if there are many rational curves on the
corresponding complex variety $X(\c)$.
First one should consider rational varieties.
\medskip

{\it Proposition} \ref{higher.dim.say}.1.
Let $X$ be a smooth, projective, real  variety of dimension
$\geq 3$ that is rational. Then  every
simple, connected, closed curve  $L\subset X(\r)$ can be  approximated
by smooth rational curves.
\medskip

Proof. Represent $L$ as the image of an embedding
$\s^1\to X(\r)$.
The proof of \cite{bk3}
 automatically produces  approximations
by  maps  $g:\p^1\to X$ such that
$g^*T_X$ is ample. By an easy lemma
(cf.\ \cite[II.3.14]{rc-book}) a general small perturbation of
any morphism $g:\p^1\to X$ such that
$g^*T_X$ is ample is an embedding. \qed
\medskip

The next class to consider is geometrically rational varieties, or,
more generally, rationally connected varieties \cite[Chap.IV]{rc-book}.

Let $X$ be a smooth, real variety such that $X(\c)$ is rationally connected.
By  a combination of \cite[Cor.1.7]{k-loc} and \cite[Thm.23]{k-spec},
 if $p_1,\dots, p_s\in X(\r)$ are in the same connected component
then there is a rational curve  $g:\p^1\to X$ passing through
all of them.  By the previous argument,
we can even choose $g$ to be an embedding if $\dim X\geq 3$.
Thus $X$ contains plenty of smooth rational curves.

Nonetheless, we believe that usually not every homotopy class
of $X(\r)$ can be represented by rational curves.
The following example illustrates some of the possible obstructions.
\medskip

{\it Example} \ref{higher.dim.say}.2.
Let $q_1, q_2, q_3$ be quadrics such that $C:=(q_1=q_2=q_3=0)\subset \p^4$
 is a smooth curve
with $C(\r)\neq\emptyset$. Consider the family of 3--folds
$$
X_t:= \bigl(q_1^2+q_2^2+q_3^2-t\bigl(x_0^4+\cdots + x_4^4)=0\bigr)\subset \p^4
$$
For $0<t\ll 1$, the real points  $X_t(\r)$ form an ${\mathbb S}^2$-bundle
over $C(\r)$.
We conjecture that if $0<t\ll 1$,  then  every rational curve
$g\colon\p^1\to X_t$ gives a null-homotopic map $g\colon\r\p^1\to X_t(\r)$.

We do not know how to prove this, but the following argument shows
that if $g_t:\p^1\to X_t$ is a continuous family of rational curves defined for
every $0<t\ll 1$,  then $g_t:\r\p^1\to X_t(\r)$ is null-homotopic.
More precisely, the images $g_t(\r\p^1)\subset X_t(\r)$
shrink to a point as $t\to 0$.

Indeed, otherwise by taking the limit as $t\to 0$, we get a
non-constant map  $\tilde g_0:\p^1\to C$.
However, the genus of $C$ is 5, hence every map $\p^1\to C$
is constant.
(A priori, the limit, taken in the moduli space of stable maps
as in \cite{ful-pan},
is a morphism $g_0:B\to X_0$ where $B$ is a (usually reducible)
real curve with only nodes as singularities such that $h^1(B,\o_B)=0$.
For such curves, the set of real points $B(\r)$ is a connected set.
Thus  the image of $B(\r)$
is a connected subset of $X_0(\r)$ that contains
 at least 2 distinct  points.
 Since $X_0(\r)=C(\r)$, one of the irreducible components of $B$ gives a
non-constant map  $\tilde g_0:\p^1\to C$.)

Unfortunately, this only implies that if we have a sequence
$t_i\to 0$ and a sequence of homotopically nontrivial
rational curves $g_{t_i}:\p^1\to X$ then their degree must go to
infinity.  We did not exclude the possibility that, as $t_i\to 0$,
we have higher and higher
degree maps approximating non null-homotopic loops.

We do not have a conjecture about which homotopy classes give
obstructions. On the other hand,
while we do not have much evidence, the following could be true.
\medskip

{\it Conjecture} \ref{higher.dim.say}.3. Let $X$ be a smooth,
rationally connected variety defined over $\r$.
Then a  $C^{\infty}$ map ${\mathbb S}^1\to X(\r)$
can be approximated by rational
 curves iff it is homotopic to a rational curve
$\r\p^1\to X(\r)$.
\end{say}

\def\cprime{$'$} \def\cprime{$'$} \def\cprime{$'$} \def\cprime{$'$}
  \def\cprime{$'$} \def\cprime{$'$} \def\dbar{\leavevmode\hbox to
  0pt{\hskip.2ex \accent"16\hss}d} \def\cprime{$'$} \def\cprime{$'$}
  \def\polhk#1{\setbox0=\hbox{#1}{\ooalign{\hidewidth
  \lower1.5ex\hbox{`}\hidewidth\crcr\unhbox0}}} \def\cprime{$'$}
  \def\cprime{$'$} \def\cprime{$'$} \def\cprime{$'$}
  \def\polhk#1{\setbox0=\hbox{#1}{\ooalign{\hidewidth
  \lower1.5ex\hbox{`}\hidewidth\crcr\unhbox0}}} \def\cdprime{$''$}
  \def\cprime{$'$} \def\cprime{$'$} \def\cprime{$'$} \def\cprime{$'$}
\providecommand{\bysame}{\leavevmode\hbox to3em{\hrulefill}\thinspace}
\providecommand{\MR}{\relax\ifhmode\unskip\space\fi MR }
\providecommand{\MRhref}[2]{%
  \href{http://www.ams.org/mathscinet-getitem?mr=#1}{#2}
}
\providecommand{\href}[2]{#2}

\end{document}